%% file: Example_of_non_trivial_sing_SSF.tex
\documentclass[12pt,oneside]{amsart} 
\usepackage{amsfonts,amssymb,amsmath,euscript}
\input Macros
\input MyMakeLit

\ndef{\mbCz}{\mbC^{(z)}} \ndef{\mbCr}{\mbC^{(r)}}

\ndef{\yy}{y}

\ndef{\ells}{{\ell_2}} \ndef{\hlambda}{{\mathfrak h_\lambda}}
\ndef{\hlambdao}{{\mathfrak h_\lambda^{(0)}}}
\ndef{\hlambdar}{{\mathfrak h_\lambda^{(r)}}}

\begin{document}
\title[Singular SSF's exist]{Non-trivial singular spectral shift \\ functions exist}
\author{\Azamov}
\address{School of Computer Science, Engineering and Mathematics
   \\ Flinders University
   \\ Bedford Park, 5042, SA Australia.}
\email{azam0001@csem.flinders.edu.au}

\begin{abstract} In this paper I prove existence of an irreducible pair of
operators $H_0$ and $H_0+V,$ where $H_0$ is a self-adjoint operator and $V$ is a self-adjoint trace-class operator,
such that the singular spectral shift function
$\xis_{H_0+V,H_0}$ of the pair is non-zero on the absolutely
continuous spectrum of $H_0.$
\end{abstract}
\maketitle

\section{Introduction}
Let $H_0$ be a self-adjoint operator and $V$ be a trace-class self-adjoint operator.
The Lifshits-Krein spectral shift function $\xi$ is the unique $L_1$-function on $\mbR$
such that for any $f \in C^\infty(\mbR)$ with compact support the trace formula
$$
  \Tr(f(H_0+V)-f(H_0)) = \int f'(\lambda)\xi(\lambda)\,d\lambda
$$
holds. Existence of such a function was proved in \cite{Kr53MS}.
Further, in \cite{BS75SM}, M.\,Sh.\,Birman and M.\,Z.\,Solomyak established the formula
$$
  \xi(\lambda) = \frac d{d\lambda} \int _0^1 \Tr\brs{V E_{(-\infty,\lambda]}(H_r) }\,dr
$$
for the spectral shift function; this formula is called spectral averaging formula.

In \cite{BK62DAN}, M.\,Sh.\,Birman and M.\,G.\,Krein proved that for a.e. $\lambda \in \mbR$
the formula
\begin{equation} \label{F: det S=exp(xi)}
  \det S(\lambda; H_0+V, H_0) = e^{-2\pi i \xi(\lambda)}
\end{equation}
holds, where $S(\lambda; H_0+V, H_0)$ is the so-called scattering matrix of the pair $H_0+V$ and $H_0$
(see e.g. \cite{Ya,Az3v4}).

In \cite{Az} I showed that for Shr\"odinger operators with sufficiently regular potentials
the following variant of the Birman-Krein formula holds:
\begin{equation} \label{F: det S=exp(xia)}
  \det S(\lambda; H_0+V, H_0) = e^{-2\pi i \xia(\lambda)}.
\end{equation}
The function $\xia(\lambda),$ which in \cite{Az} was called the absolutely continuous part of the spectral shift function,
can be given by the formula
$$
  \xia(\lambda) := \frac d{d\lambda} \int _0^1 \Tr\brs{V E_{(-\infty,\lambda]}(H_r^{(a)}) }\,dr,
$$
where $E_{(-\infty,\lambda]}(H)$ denotes the spectral projection of a self-adjoint operator $H,$ $H_r:=H_0+rV,$ $r \in \mbR,$
and $H_r^{(a)}$ is the absolutely continuous part of $H_r.$ The formula (\ref{F: det S=exp(xia)}) compared with the Birman-Krein formula
(\ref{F: det S=exp(xi)}) obviously implies that the function
\begin{equation} \label{F: def of xis}
  \xis(\lambda) := \frac d{d\lambda} \int _0^1 \Tr\brs{V E_{(-\infty,\lambda]}(H_r^{(s)}) }\,dr,
\end{equation}
called in \cite{Az} the singular part of the spectral shift function (or just singular spectral shift function),
is a.e. integer-valued. Here $H_r^{(s)}$ denotes the singular part of $H_r.$
For Schr\"odinger operators from the class, considered in \cite{Az}, this is a trivial result:
those operators do not have singular spectrum on the positive semi-axis, so that $\xis = 0$
on $[0,\infty),$ while on $(-\infty,0]$ the function $\xis$ coincides with $\xi,$ which is well-known to be integer-valued on $(-\infty,0].$

In \cite{Az3v4} the formula (\ref{F: det S=exp(xia)}) was proved for arbitrary self-adjoint operators $H_0$ and arbitrary self-adjoint trace-class
perturbations $V.$ Combined with the Birman-Krein formula, this implies that for any such operators the singular spectral shift function is a.e. integer-valued.
Looking at the definition (\ref{F: def of xis}) of the singular spectral shift function,
this is quite an unexpected result (in my opinion). But still, there were no examples of non-trivial singular spectral shift functions.

By a non-trivial example of singular spectral shift function I mean an example of a self-adjoint operator $H_0$ and a self-adjoint trace-class operator $V,$
such that $H_0$ and $V$ do not have common non-trivial invariant subspace and such that the restriction of the singular spectral shift function for the pair $H_0+V$ and $H_0$
to the absolutely continuous spectrum of $H_0$ is non-zero. Trivial examples are easily constructed: one can take the direct sum of any pair with a finite-dimensional pair,
--- if the support of the spectral shift function of the finite-dimensional pair overlaps the absolutely continuous spectrum of another pair, we are done.

In this paper I prove existence of a non-trivial singular spectral
shift function. The proof clearly indicates that non-trivial singular spectral
shift functions not only exist, --- they are ubiquitous. At the same time, it is not easy to point out
to a concrete pair $H_0$ and $V$ with non-trivial singular spectral shift function, and there seems to be a reason for this.

\section{Results}
\subsection{Preliminaries}
Let $H$ be a self-adjoint operator on a (complex separable) Hilbert space $\hilb$ and let $\set{E_\lambda}$ be its spectral decomposition.
Recall that a vector $f \in \hilb$ is absolutely continuous (respectively, singular, pure point, singular continuous) with respect to $H,$
if the measure with distribution function $\scal{E_\lambda f}{f}$ is absolutely continuous (respectively, singular, pure point, singular continuous).
The set of all absolutely continuous (respectively, singular, pure point, singular continuous) vectors form a (closed) subspace of $\hilb,$
which we denote by $\hilb^{(a)}$ (respectively, $\hilb^{(s)},$ $\hilb^{(pp)},$ $\hilb^{(sc)}$). The subspaces
$\hilb^{(a)},$  $\hilb^{(s)},$ $\hilb^{(pp)},$ $\hilb^{(sc)}$ are mutually orthogonal and invariant with respect to $H;$
also, the following decompositions hold: $\hilb = \hilb^{(a)} \oplus \hilb^{(s)}$ and $\hilb^{(s)} = \hilb^{(pp)} \oplus \hilb^{(sc)}.$
Restriction of $H$ to $\hilb^{(a)}$ (respectively, to $\hilb^{(s)},$ $\hilb^{(pp)},$ $\hilb^{(sc)}$) will be denoted by $H^{(a)}$ (respectively, by $H^{(s)},$ $H^{(pp)},$ $H^{(sc)}$).
The equalities
$$
  H = H^{(a)} \oplus H^{(s)} \ \ \text{and} \ \ H^{(s)} = H^{(pp)} \oplus H^{(sc)}
$$
hold. See e.g. \cite{Ya} for details.

Let $H_0$ be a self-adjoint operator and let $V$ be a trace-class self-adjoint operator.
We introduce the Lifshits-Krein spectral shift function $\xi_{H_0+V,H_0}$ of the pair $H_0,H_0+V$
by the Birman-Solomyak spectral averaging formula, as a distribution,
$$
  \xi(\phi) = \int_0^1 \Tr(V \phi(H_r))\,dr, \quad \phi \in C_c^\infty,
$$
where $H_r := H_0+rV,$ $r \in \mbR.$
The distribution $\xi$ is an absolutely continuous finite measure (see \cite{Kr53MS}).
Similarly, we introduce the absolutely continuous $\xia_{H_0+V,H_0}$
and singular $\xis_{H_0+V,H_0}$ parts of the spectral shift function by formulas
$$
  \xia(\phi) = \int_0^1 \Tr(V \phi(H_r^{(a)}))\,dr, \qquad  \xis(\phi) = \int_0^1 \Tr(V \phi(H_r^{(s)}))\,dr,
$$
where $H^{(a)}$ and $H^{(s)}$ denote the absolutely continuous and singular parts of a self-adjoint operator $H$ respectively.
It is shown in \cite{Az3v4} that $\xia$ and $\xis$ are also absolutely continuous finite measures (see \cite[Lemma 9.7]{Az3v4}).
Plainly,
$$
  \xi_{H_0+V,H_0} = \xia_{H_0+V,H_0} + \xis_{H_0+V,H_0}.
$$

One can also introduce the pure point and singular continuous parts of the spectral shift function,
by similar formulas:
$$
  \xi^{(pp)}(\phi) = \int_0^1 \Tr(V \phi(H_r^{(pp)}))\,dr, \qquad  \xi^{(sc)}(\phi) = \int_0^1 \Tr(V \phi(H_r^{(sc)}))\,dr.
$$
Plainly,
$$
  \xis_{H_0+V,H_0} = \xi^{(pp)}_{H_0+V,H_0} + \xi^{(sc)}_{H_0+V,H_0}.
$$
If $V \geq 0,$ then it is easy to see that $\xi^{(pp)}$ and $\xi^{(sc)}$ are also absolutely continuous.
Whether they are absolutely continuous for arbitrary trace-class $V,$ I don't know.

In \cite{Kr53MS} it was proved that $\xi$ is additive; that is, for any three self-adjoint operators $H_0,H_1,H_2$ with common domain
and trace-class differences the equality
$$
  \xi_{H_2,H_0} = \xi_{H_2,H_1} + \xi_{H_1,H_0}
$$
holds. In \cite{Az5} it was proved that $\xia$ and $\xis$ are also additive (see \cite[Theorem 2.2 and Corollay 2.3]{Az5}):
$$
  \xia_{H_2,H_0} = \xia_{H_2,H_1} + \xia_{H_1,H_0}, \qquad  \xis_{H_2,H_0} = \xis_{H_2,H_1} + \xis_{H_1,H_0}.
$$

\subsection{Idea of the proof}
In case of operators $H_0$ and $H_0+V$ on a finite dimensional Hilbert space, the spectral shift function $\xi(\lambda)$
is equal to the total number of eigenvalues of the family of operators $H_r = H_0+rV,$ which cross $\lambda$ in the right direction
minus the total number of eigenvalues which cross $\lambda$ in the left direction.
In finite-dimensional situation there is only pure point spectrum, so obviously $\xi(\lambda) = \xis(\lambda).$
Now, we take any other infinite-dimensional pair of operators $H_0'$ and $H_0'+V'$ and take the direct sum of those two pairs:
$H_0\oplus H_0'$ and $(H_0 + V) \oplus (H_0' + V').$ Obviously,
$$
  \xis_{(H_0 + V) \oplus (H_0' + V'), H_0\oplus H_0'} = \xis_{H_0 + V, H_0} + \xis_{H_0' + V', H_0'},
$$
with analogous equalities for all other parts of the spectral shift function. In this way we get a pair of infinite-dimensional
operators $H_0$ and $H_1$ with non-zero singular spectral shift function, but the problem is that this pair is not irreducible.
By the additivity property of the singular spectral shift function, for any other self-adjoint $H_2$ such that $H_2-H_0$ is trace-class,
we have the equality
$$
  \xis_{H_1,H_0} = \xis_{H_1,H_2} + \xis_{H_2,H_0}.
$$
It follows that if $\xis_{H_1,H_0}$ does not vanish on the absolutely continuous spectrum $\sigma^{(a)}$ of $H_0,$ then at least one of the functions
$\xis_{H_1,H_2}$ and $\xis_{H_2,H_0}$ also does not vanish on $\sigma^{(a)}.$
(Note that since differences $H_i-H_j$ are trace-class, their absolutely continuous spectra coincide).
Now, one has to ensure that both pairs $(H_0,H_2)$ and $(H_1,H_2)$
are either irreducible or that $\xis$ of an irreducible part of a pair does not vanish on the absolutely continuous spectrum of the reduced operators.

In accomplishing of this plan we have a big freedom of choice of operators due to the additivity property of the singular spectral shift function.

\subsection{Results}
Let $\tilde \hilb = L_2(\mbR)$ and let
$\hilb = \tilde \hilb \oplus \mbC.$

An operator $A$ on $\hilb$ can be represented in the form of a
matrix
$$
  A = \left(\begin{array}{cc} \tilde A & f_1 \\ \scal{f_2}{\cdot} & z_0
  \end{array}\right),
$$
where $\tilde A$ is an operator on $\tilde \hilb,$ $f_1,f_2 \in \tilde
\hilb$ and $z_0\in \mbC.$ It is not difficult to see that $A$ is
self-adjoint if and only if $\tilde A$ is self-adjoint, $f_1=f_2$ and
$z_0$ is real. So, a self-adjoint operator on $\hilb$ has the form
$$
  H = \left(\begin{array}{cc} D & f \\ \scal{f}{\cdot} & a_0
  \end{array}\right),
$$
where $D=D^*$ and $a_0 \in \mbR.$

Let $a \in \mbR$ and $v \in \tilde \hilb.$
The self-adjoint operator
\begin{equation} \label{F: Va}
  V_a = \left(\begin{array}{cc} v\scal{v}{\cdot} & v \\ \scal{v}{\cdot} & a
  \end{array}\right)
\end{equation}
has rank less or equal to $2.$ Further, let
\begin{equation} \label{F: D}
  D = \frac 1i \frac{d}{dx};
\end{equation}
the operator $D$ is a self-adjoint operator on $\tilde \hilb,$
its spectrum is absolutely continuous and is equal to $\mbR.$
Let also
\begin{equation} \label{F: H0}
  H_0 = \left(\begin{array}{cc} D & 0 \\ 0 & -1
  \end{array}\right)
\ \ \text{and}  \ \
  V = \left(\begin{array}{cc} 0 & 0 \\ 0 & 2
  \end{array}\right).
\end{equation}
In this case $H_r := H_0+rV$ is equal to
\begin{equation} \label{F: H1}
  H_r = \left(\begin{array}{cc} D & 0 \\ 0 & -1 + 2r
  \end{array}\right).
\end{equation}
It is not difficult to see that
$$
  \xi_{H_1,H_0} = \xis_{H_1,H_0} = \chi_{[-1,1]}.
$$

Let $\tilde V = v \scal{v}{\cdot},$
where from now on the vector $v \in \tilde \hilb = L_2(\mbR)$ is given by
\begin{equation} \label{F: v}
  v(x) = \frac 1{\sqrt[4]{\pi}}e^{-\frac{x^2}{2}};
\end{equation}
the coefficient is chosen so that $\norm{v} = 1.$
It is well known that the Fourier transform of $v$ is equal to $v.$
The operator $\tilde V$ is one-dimensional and so it is trace-class.

We say that a (closed) subspace $\clK$ of $\hilb$ is invariant for an operator $H$ on a Hilbert space $\hilb,$
if $\clK \cap \euD(H)$ is dense in $\clK$ and if $Hf \in \clK$ for all $f \in \clK \cap \euD(H).$
We say that a pair of operators $H_1$ and $H_2$ is irreducible, if the only subspaces of $\hilb$
which are invariant for both $H_1$ and $H_2$ are $\set{0}$ and $\hilb.$

\begin{lemma} \label{L: (D,V) is irre-ble} The pair $(D, \tilde V)$ is irreducible.
\end{lemma}
\begin{proof}
Let $\clK$ be a subspace of $\tilde \hilb$ which is invariant with respect to both $D$ and $\tilde V$
and let $\clK \neq \hilb.$

(A) Claim: $v \notin \clK.$
It is known that the vector $v$ is cyclic for $D$ ---
Hermite polynomials which form a basis of $\tilde \hilb$ are linear combinations of $D^k v.$
It follows that $v \notin \clK.$

(B) Claim: $v \perp \clK.$
Indeed, if $f \in \clK$ then $Vf = \scal{v}{f}v \in \clK.$
If additionally $f \neq 0$ and if $v$ is not orthogonal to $f,$ then
this implies that $v \in \clK.$ This contradicts (A).

(C) Claim: for any $n=0,1,2,\ldots$ \ $D^n v \perp \clK.$

Proof. By (B), for $n=0$ this is true. Assume that $D^n v \perp \clK$ for $n = k.$
Let $f \in \clK \cap \euD(D).$ Then $Df \in \clK,$ so by the assumption
$\scal{Df}{D^k v} = 0.$ It follows that $\scal{f}{D^{k+1} v} = 0.$
Since $\clK \cap \euD(D)$ is dense in $\clK,$ this implies that $D^{k+1} v \perp \clK.$
The proof of (C) is complete.

(D) Since the vector $v$ is cyclic for $D,$ it follows from (C) that $\clK = \set{0}.$

The proof is complete.
\end{proof}


We consider operators $V_1$ and $V_{-1},$ given by (\ref{F: Va}).

It is not difficult to check that $V_1$ has rank one with eigenvector ${v \choose 1}$
and eigenvalue $2,$ and that $V_{-1}$ has rank two with eigenvectors
$$
  {v \choose -1 \pm \sqrt{2}}
$$
and eigenvalues $\pm \sqrt 2.$ Note also, that
$$
  H_0 + V_1 =  H_1 + V_{-1}.
$$

\begin{lemma} \label{L: pair (H0,V1) is irreducible} Let $H_0$ be the operator (\ref{F: H0}), where $D$ is given by (\ref{F: D}),
and let $V_1$ be the operator given by (\ref{F: Va}), where $v$ is given by (\ref{F: v}). Let $\mathbf v = {v \choose 1}.$
Let $\clK$ be the subspace of $\tilde \hilb$ generated by vectors $H_0^k \mathbf v,$ $k = 0,1,2,\ldots$ and
let $H':=H_0 \big|_{\clK}$ and $V':=V_1 \big|_{\clK}$ be restrictions
of operators $H_0$ and $V_1$ to $\clK$ (it is easily seen that $\clK$ is invariant with respect to both $H_0$ and $V_1$).
The pair $(H',V')$ is irreducible.
\end{lemma}
\begin{proof}
Assume the contrary: let $\clL$ be a non-trivial subspace of $\clK$ which is invariant with respect to both $H_0$ and $V_1.$
So, let $\mathbf f$ be a non-zero vector from $\clL.$

If $V_1 \mathbf f$ is non-zero, then the vector $\mathbf v$ belongs to $\clL,$
since $V_1 \mathbf f = \alpha \mathbf v$ for some $\alpha \in \mbC.$ Since $\mathbf v$ is cyclic for $H'$ (by definition), it follows that $\clL = \clK.$
This shows that $V_1 \mathbf f = 0$ for any $\mathbf f \in \clL.$

Since $\clL$ is invariant with respect to $H',$ the subspace $\clL \cap \euD(H')$ is dense in $\clL.$
Let $\tilde \clL$ be the maximal domain of the restriction of $H'$ to $\clL.$

Now, let $\mathbf f \in \bigcap_{k=1}^\infty \euD(H'\big|_{\tilde \clL}^k)$ be a non-zero vector (such a vector exists).
The vector $\mathbf f$ can be assumed to be either in the form ${f \choose 0}$ or ${f \choose 1}.$

If $\mathbf f = {f \choose 0},$ then $H_0^k \mathbf f = {D^k f \choose 0} \in \clL$ for all $k=0,1,2,\ldots;$
since $V_1$ vanishes on $\clL,$ it follows
that $V_1 (D^k f,0) = 0$ for all $k;$ this implies that $\scal{v}{D^k f} = 0,$ which yields
$\scal{D^k v}{f} = 0;$ since $v$ is cyclic, it follows that $f = 0.$ This is a contradiction.

So, let $\mathbf f = {f \choose 1}.$
Since all vectors $H'^k \mathbf f$ belong to $\clL,$ it follows that all vectors
$$
  { D^k f \choose (-1)^k }
$$
belong to the kernel of $V_1.$ It follows from $V_1{D^k f \choose (-1)^k}$ that
$$
  \scal{D^k v}{f} = \scal{v}{D^k f} = (-1)^{k+1}
$$
for all $k=0,1,2,\ldots.$ It follows that
$$
  \scal{D^k v}{(1+D)f} = 0
$$
for all $k=0,1,2,\ldots.$ Since $v$ is cyclic for $D,$ it follows that $(1-D) f = 0.$
This implies that $f = 0.$ This implies that $(0,1) \in \tilde \clL,$ which in its turn implies that $V_1 (0,1) = 0,$ so that $v = 0.$ This is a contradiction.

The proof is complete.
\end{proof}

\begin{lemma} \label{L: pair (H1,V2) is irreducible} Let $H_1$ be the operator (\ref{F: H1}) (with $r=1$), where $D$ is given by (\ref{F: D});
let $V_{-1}$ be the operator given by (\ref{F: Va}) with $a=-1,$ where $v$ is given by (\ref{F: v}).
The pair $(H_1,V_{-1})$ is irreducible.
\end{lemma}
\begin{proof} (A) Assume the contrary: let $\clK$ be a non-trivial subspace of $\hilb,$ which is invariant with respect to both $H_1$
and $V_{-1}.$

Note that the proof will be complete, if we show that one of the vectors ${v \choose 0}$ or ${0 \choose 1}$ belongs to $\clK.$
Indeed, in this case the invariance of $\clK$ with respect to $V_{-1}$ would imply that the other vector also belongs to $\clK,$
and cyclicity of $v$ with respect to $D$ would complete the proof.

(B) Let $\mathbf f = {f \choose f_0} \in \clK$ be a non-zero vector.
Claim: we can assume that $f_0 = 1.$ Indeed, assume the contrary: $\mathbf f = {f \choose 0}.$
Since $\clK$ is invariant with respect to $H_1,$ it follows that ${D^k f \choose 0} \in \clK$ for all $k=0,1,\ldots.$
Since also $V_{-1} \clK \subset \clK,$ this implies that $\scal{v}{D^k f} = 0.$ It follows that $\scal{D^k v}{f} = 0$ for all
$k=0,1,\ldots.$ Since $v$ is cyclic for $D,$ it follows that $f = 0.$ The claim is proved.

(C) So, let $\mathbf f = {f \choose 1} \in \clK.$
It follows that
$$
  V_{-1} \mathbf f = { (\scal{v}{f} + 1)v \choose \scal{v}{f} - 1 } \in \clK.
$$
If one of the numbers $\scal{v}{f} \pm 1$ is zero, then the proof is complete by (A).

So, we can assume that $\scal{v}{f} \pm 1 \neq 0.$

It follows that ${v \choose \alpha} \in \clK,$ where $\alpha = \frac{\scal{v}{f} - 1}{\scal{v}{f} + 1}.$
Since $V_{-1} {v \choose \alpha} = (1+\alpha)^{-1} {v \choose \frac{1-\alpha}{1+\alpha}}\in \clK$ too, if $\alpha \neq -1 \pm \sqrt{2}$
(that is, if ${v \choose \alpha}$ is not an eigenvector of $V_{-1}$),
it follows that ${v \choose 0} \in \clK,$ and (A) completes the proof.

So, we can assume that ${v \choose \alpha} \in \clK$ is an eigenvector of $V_{-1}.$ Since $H_1{v \choose \alpha} \in \clK,$
it follows that ${Dv \choose \alpha} \in \clK.$ Further,
$$
  V_{-1}{Dv \choose \alpha} = {(\scal{v}{Dv}+\alpha)v \choose \scal{v}{Dv}-\alpha} \in \clK.
$$
It is easy to check that this vector and ${v \choose \alpha} \in \clK$ are linearly independent; hence, the proof is complete by (A).
\end{proof}

\begin{thm} There exists an irreducible pair consisting of a self-adjoint operator $H_0$ with non-empty absolutely continuous spectrum
and a trace-class self-adjoint operator $V,$ such that the restriction of the singular
spectral shift function $\xis_{H_0+V,H_0}$ of the pair $(H_0,H_0+V)$ to the absolutely continuous spectrum of $H_0$
is non-zero.
\end{thm}
\begin{proof} Let $H_0,$ $V_1$ and $V_{-1}$ be as in Lemmas \ref{L: pair (H0,V1) is irreducible} and \ref{L: pair (H1,V2) is irreducible}.
Additivity of the singular spectral shift \cite[Corollary 2.3]{Az5} implies that
\begin{equation} \label{F: xis=xis+xis}
  \xis_{H_0+V_1,H_0} + \xis_{H_1,H_1+V_{-1}} = \xis_{H_1,H_0} = \chi_{[-1,1]}.
\end{equation}
So, at least one of the functions $\xis_{H_0+V_1,H_0}$ or $\xis_{H_1,H_1+V_{-1}}$ is non-zero
on $[-1,1],$ which is a subset of the absolutely continuous spectrum of $H_0$ (and of $H_1$).

If $\xis_{H_1,H_1+V_{-1}} \neq 0$ on $[-1,1],$ then we are done by Lemma \ref{L: pair (H1,V2) is irreducible}.

If $\xis_{H_1,H_1+V_{-1}} = 0$ on $[-1,1],$ then we note that the complement of the subspace $\clK$
generated by vectors $H_0^k \mathbf v,$ $k=0,1,2,\ldots,$ has dimension $\leq 1;$
the subspace $\clK$ is invariant with respect to both $H_0$ and $V_1,$ and by Lemma \ref{L: pair (H1,V2) is irreducible} the restrictions
$H_0'$ and $V_1'$ of operators $H_0$ and $V_1$ to the subspace $\clK$ give an irreducible pair. Now note that the restriction of $V_1$
to $\clK^\perp$ is zero, so that
$$
  \xis_{H_0'+V_1',H_0'} = \xis_{H_0+V_1,H_0} = \chi_{[-1,1]}.
$$
Further, since $\clK^\perp$ is at most one-dimensional, the difference $H_0-H_0'$ is a finite-rank operator.
This implies that the absolutely continuous spectrum of $H_0'$ coincides with that of $H_0,$ which is $\mbR.$

The proof is complete.
\end{proof}

\begin{lemma} \label{L: pure point is empty} If $r>0,$ then for any $\alpha \in \mbR$ the pure point spectrum of the operator
$$
  H := \left(\begin{array}{cc} D + r\scal{v}{\cdot}v & rv \\ r\scal{v}{\cdot} & \alpha
  \end{array}\right),
$$
where $D$ is given by (\ref{F: D}) and $v$ is given by (\ref{F: v}), is empty.
\end{lemma}
\begin{proof}
Assume that there is a non-zero vector $\mathbf f = { f \choose f_0 } \in \hilb$
such that $H \mathbf f = \lambda \mathbf f$ for some $\lambda \in \mbR.$
This implies that $f$ belongs to the domain of $D.$ Further, we have
$$
  H\mathbf f = \left(\begin{array}{c} D f + r\scal{v}{f}v + rf_0 v \\ r\scal{v}{f} + \alpha f_0
  \end{array}\right) = \left(\begin{array}{c} \lambda f \\ \lambda f_0 \end{array}\right).
$$
This implies that
$$
  D f = \lambda f - r\scal{v}{f} v - r f_0 v,
$$
so that $f' \in L_2(\mbR).$ Taking the Fourier transform of the last equality gives
$$
  \xi \hat f(\xi) = \lambda \hat f(\xi) - r\scal{v}{f} \hat v(\xi) - r f_0 \hat v(\xi).
$$
Since $\hat v = v,$ it follows that
$$
  \hat f(\xi) = -r\brs{\scal{v}{f} + f_0} \cdot \frac{v(\xi)}{\xi - \lambda}.
$$
Since $\frac{v(\xi)}{\xi - \lambda}$ is not $L_2,$ it follows that $f = 0$ and $f_0=0;$ that is, $\mathbf f=0.$

This contradiction completes the proof.
\end{proof}

By Lemma \ref{L: pure point is empty},
the pure point spectrum of operators $H_0+rV_1$ and $H_1+rV_{-1},$ $r \in (0,1],$ is empty.
It follows that for both pairs of operators
\begin{equation} \label{F: xipp=0}
  \xi^{(pp)}_{H_0+rV_1,H_0} = 0 \quad \text{and} \quad \xi^{(pp)}_{H_1,H_1+rV_{-1}} = 0.
\end{equation}
Consequently, by (\ref{F: xis=xis+xis}),
$$
  \xi^{(sc)}_{H_0+V_1,H_0} + \xi^{(sc)}_{H_1,H_1+V_{-1}} = \chi_{[-1,1]},
$$
so that
$$
  \xi^{(sc)}_{H_0+V_1,H_0} \neq 0 \quad \text{or} \quad \xi^{(sc)}_{H_1,H_1+V_{-1}} \neq 0.
$$
\begin{cor} The pure point and the singular continuous spectral shift distributions are not additive.
\end{cor}
\begin{proof} It follows from (\ref{F: xipp=0}) and (\ref{F: xis=xis+xis}) that pure point spectral shift function is not additive.
This fact, combined with additivity of the singular spectral shift function, implies non-additivity of the singular continuous spectral shift function.
\end{proof}
This corollary shows a significant difference between absolutely continuous and singular spectral shift functions, on one hand,
and pure point and singular continuous spectral shift functions on the other hand.


Note that the results of this paper are in full accordance with the generic property of the singular
continuous spectrum which was observed in \cite{RJMS,RMS} (see also \cite{SimTrId2} and \cite{SW}.

The example of a non-trivial singular spectral shift function presented here is a clear indication of the fact that
non-trivial singular spectral shift functions are ubiquitous, even though it is not easy to present them explicitly;
they are a bit like transcendental numbers.

Note that eigenvalues of $H_r$ in the absolutely continuous spectrum can pop up suddenly and disappear in the same way as $r$ changes (see some examples in \cite{SimTrId2}), but it seems that
they do not contribute to the singular part of the spectral shift function.
In \cite{Az3v4} I made a conjecture that in the case of trace-class perturbations $V$ the pure point part of the
spectral shift function must be zero on the absolutely continuous spectrum (which is the same for all operators $H_r$ in the path).

\input MyListOfRef

\end{document}

%% file: Macros.tex
    \newtheorem{thm}{Theorem}                     [section]
    \newtheorem{thm*}{Theorem}
    
    \newtheorem{lemma}[thm]{Lemma}
    \newtheorem{cor}[thm]{Corollary}

    \newtheorem{lemma*}{Lemma}    

    \newtheorem{rems*}{Remark}   

\newcommand{\ndef}{\newcommand*}
\def\rndef{\renewcommand}

\ndef{\myaddress}[1]{\begin{center} \it\small #1 \end{center}}

\ndef{\clA}{{\mathcal A}} \ndef{\rmA}{{\mathrm A}} \ndef{\mbA}{{\mathbb A}} \ndef{\bfA}{{\mathbf A}} \ndef{\euA}{{\EuScript A}} \ndef{\frA}{{\mathfrak A}}
\ndef{\clB}{{\mathcal B}} \ndef{\rmB}{{\mathrm B}} \ndef{\mbB}{{\mathbb B}} \ndef{\bfB}{{\mathbf B}} \ndef{\euB}{{\EuScript B}} \ndef{\frB}{{\mathfrak B}}
\ndef{\clC}{{\mathcal C}} \ndef{\rmC}{{\mathrm C}} \ndef{\mbC}{{\mathbb C}} \ndef{\bfC}{{\mathbf C}} \ndef{\euC}{{\EuScript C}} \ndef{\frC}{{\mathfrak C}}
\ndef{\clD}{{\mathcal D}} \ndef{\rmD}{{\mathrm D}} \ndef{\mbD}{{\mathbb D}} \ndef{\bfD}{{\mathbf D}} \ndef{\euD}{{\EuScript D}} \ndef{\frD}{{\mathfrak D}}
\ndef{\clE}{{\mathcal E}} \ndef{\rmE}{{\mathrm E}} \ndef{\mbE}{{\mathbb E}} \ndef{\bfE}{{\mathbf E}} \ndef{\euE}{{\EuScript E}} \ndef{\frE}{{\mathfrak E}}
\ndef{\clF}{{\mathcal F}} \ndef{\rmF}{{\mathrm F}} \ndef{\mbF}{{\mathbb F}} \ndef{\bfF}{{\mathbf F}} \ndef{\euF}{{\EuScript F}} \ndef{\frF}{{\mathfrak F}}
\ndef{\clG}{{\mathcal G}} \ndef{\rmG}{{\mathrm G}} \ndef{\mbG}{{\mathbb G}} \ndef{\bfG}{{\mathbf G}} \ndef{\euG}{{\EuScript G}} \ndef{\frG}{{\mathfrak G}}
\ndef{\clH}{{\mathcal H}} \ndef{\rmH}{{\mathrm H}} \ndef{\mbH}{{\mathbb H}} \ndef{\bfH}{{\mathbf H}} \ndef{\euH}{{\EuScript H}} \ndef{\frH}{{\mathfrak H}}
\ndef{\clI}{{\mathcal I}} \ndef{\rmI}{{\mathrm I}} \ndef{\mbI}{{\mathbb I}} \ndef{\bfI}{{\mathbf I}} \ndef{\euI}{{\EuScript I}} \ndef{\frI}{{\mathfrak I}}
\ndef{\clJ}{{\mathcal J}} \ndef{\rmJ}{{\mathrm J}} \ndef{\mbJ}{{\mathbb J}} \ndef{\bfJ}{{\mathbf J}} \ndef{\euJ}{{\EuScript J}} \ndef{\frJ}{{\mathfrak J}}
\ndef{\clK}{{\mathcal K}} \ndef{\rmK}{{\mathrm K}} \ndef{\mbK}{{\mathbb K}} \ndef{\bfK}{{\mathbf K}} \ndef{\euK}{{\EuScript K}} \ndef{\frK}{{\mathfrak K}}
\ndef{\clL}{{\mathcal L}} \ndef{\rmL}{{\mathrm L}} \ndef{\mbL}{{\mathbb L}} \ndef{\bfL}{{\mathbf L}} \ndef{\euL}{{\EuScript L}} \ndef{\frL}{{\mathfrak L}}
\ndef{\clM}{{\mathcal M}} \ndef{\rmM}{{\mathrm M}} \ndef{\mbM}{{\mathbb M}} \ndef{\bfM}{{\mathbf M}} \ndef{\euM}{{\EuScript M}} \ndef{\frM}{{\mathfrak M}}
\ndef{\clN}{{\mathcal N}} \ndef{\rmN}{{\mathrm N}} \ndef{\mbN}{{\mathbb N}} \ndef{\bfN}{{\mathbf N}} \ndef{\euN}{{\EuScript N}} \ndef{\frN}{{\mathfrak N}}
\ndef{\clO}{{\mathcal O}} \ndef{\rmO}{{\mathrm O}} \ndef{\mbO}{{\mathbb O}} \ndef{\bfO}{{\mathbf O}} \ndef{\euO}{{\EuScript O}} \ndef{\frO}{{\mathfrak O}}
\ndef{\clP}{{\mathcal P}} \ndef{\rmP}{{\mathrm P}} \ndef{\mbP}{{\mathbb P}} \ndef{\bfP}{{\mathbf P}} \ndef{\euP}{{\EuScript P}} \ndef{\frP}{{\mathfrak P}}
\ndef{\clQ}{{\mathcal Q}} \ndef{\rmQ}{{\mathrm Q}} \ndef{\mbQ}{{\mathbb Q}} \ndef{\bfQ}{{\mathbf Q}} \ndef{\euQ}{{\EuScript Q}} \ndef{\frQ}{{\mathfrak Q}}
\ndef{\clR}{{\mathcal R}} \ndef{\rmR}{{\mathrm R}} \ndef{\mbR}{{\mathbb R}} \ndef{\bfR}{{\mathbf R}} \ndef{\euR}{{\EuScript R}} \ndef{\frR}{{\mathfrak R}}
\ndef{\clS}{{\mathcal S}} \ndef{\rmS}{{\mathrm S}} \ndef{\mbS}{{\mathbb S}} \ndef{\bfS}{{\mathbf S}} \ndef{\euS}{{\EuScript S}} \ndef{\frS}{{\mathfrak S}}
\ndef{\clT}{{\mathcal T}} \ndef{\rmT}{{\mathrm T}} \ndef{\mbT}{{\mathbb T}} \ndef{\bfT}{{\mathbf T}} \ndef{\euT}{{\EuScript T}} \ndef{\frT}{{\mathfrak T}}
\ndef{\clU}{{\mathcal U}} \ndef{\rmU}{{\mathrm U}} \ndef{\mbU}{{\mathbb U}} \ndef{\bfU}{{\mathbf U}} \ndef{\euU}{{\EuScript U}} \ndef{\frU}{{\mathfrak U}}
\ndef{\clV}{{\mathcal V}} \ndef{\rmV}{{\mathrm V}} \ndef{\mbV}{{\mathbb V}} \ndef{\bfV}{{\mathbf V}} \ndef{\euV}{{\EuScript V}} \ndef{\frV}{{\mathfrak V}}
\ndef{\clW}{{\mathcal W}} \ndef{\rmW}{{\mathrm W}} \ndef{\mbW}{{\mathbb W}} \ndef{\bfW}{{\mathbf W}} \ndef{\euW}{{\EuScript W}} \ndef{\frW}{{\mathfrak W}}
\ndef{\clX}{{\mathcal X}} \ndef{\rmX}{{\mathrm X}} \ndef{\mbX}{{\mathbb X}} \ndef{\bfX}{{\mathbf X}} \ndef{\euX}{{\EuScript X}} \ndef{\frX}{{\mathfrak X}}
\ndef{\clY}{{\mathcal Y}} \ndef{\rmY}{{\mathrm Y}} \ndef{\mbY}{{\mathbb Y}} \ndef{\bfY}{{\mathbf Y}} \ndef{\euY}{{\EuScript Y}} \ndef{\frY}{{\mathfrak Y}}
\ndef{\clZ}{{\mathcal Z}} \ndef{\rmZ}{{\mathrm Z}} \ndef{\mbZ}{{\mathbb Z}} \ndef{\bfZ}{{\mathbf Z}} \ndef{\euZ}{{\EuScript Z}} \ndef{\frZ}{{\mathfrak Z}}

\ndef{\tA}{{\widetilde A}} \ndef{\tcA}{{\widetilde\clA}} \ndef{\ttcA}{\widetilde{\tcA}} \ndef{\sfA}{{\textsf A}} \ndef{\ttA}{\widetilde{\tA}} \ndef{\dzA}{{A^\sharp}}
\ndef{\tB}{{\widetilde B}} \ndef{\tcB}{{\widetilde\clB}} \ndef{\ttcB}{\widetilde{\tcB}} \ndef{\sfB}{{\textsf B}} \ndef{\ttB}{\widetilde{\tB}} \ndef{\dzB}{{B^\sharp}}
\ndef{\tC}{{\widetilde C}} \ndef{\tcC}{{\widetilde\clC}} \ndef{\ttcC}{\widetilde{\tcC}} \ndef{\sfC}{{\textsf C}} \ndef{\ttC}{\widetilde{\tC}} \ndef{\dzC}{{C^\sharp}}
\ndef{\tD}{{\widetilde D}} \ndef{\tcD}{{\widetilde\clD}} \ndef{\ttcD}{\widetilde{\tcD}} \ndef{\sfD}{{\textsf D}} \ndef{\ttD}{\widetilde{\tD}} \ndef{\dzD}{{D^\sharp}}
\ndef{\tE}{{\widetilde E}} \ndef{\tcE}{{\widetilde\clE}} \ndef{\ttcE}{\widetilde{\tcE}} \ndef{\sfE}{{\textsf E}} \ndef{\ttE}{\widetilde{\tE}} \ndef{\dzE}{{E^\sharp}}
\ndef{\tF}{{\widetilde F}} \ndef{\tcF}{{\widetilde\clF}} \ndef{\ttcF}{\widetilde{\tcF}} \ndef{\sfF}{{\textsf F}} \ndef{\ttF}{\widetilde{\tF}} \ndef{\dzF}{{F^\sharp}}
\ndef{\tG}{{\widetilde G}} \ndef{\tcG}{{\widetilde\clG}} \ndef{\ttcG}{\widetilde{\tcG}} \ndef{\sfG}{{\textsf G}} \ndef{\ttG}{\widetilde{\tG}} \ndef{\dzG}{{G^\sharp}}
\ndef{\tH}{{\widetilde H}} \ndef{\tcH}{{\widetilde\clH}} \ndef{\ttcH}{\widetilde{\tcH}} \ndef{\sfH}{{\textsf H}} \ndef{\ttH}{\widetilde{\tH}} \ndef{\dzH}{{H^\sharp}}
\ndef{\tI}{{\widetilde I}} \ndef{\tcI}{{\widetilde\clI}} \ndef{\ttcI}{\widetilde{\tcI}} \ndef{\sfI}{{\textsf I}} \ndef{\ttI}{\widetilde{\tI}} \ndef{\dzI}{{I^\sharp}}
\ndef{\tJ}{{\widetilde J}} \ndef{\tcJ}{{\widetilde\clJ}} \ndef{\ttcJ}{\widetilde{\tcJ}} \ndef{\sfJ}{{\textsf J}} \ndef{\ttJ}{\widetilde{\tJ}} \ndef{\dzJ}{{J^\sharp}}
\ndef{\tK}{{\widetilde K}} \ndef{\tcK}{{\widetilde\clK}} \ndef{\ttcK}{\widetilde{\tcK}} \ndef{\sfK}{{\textsf K}} \ndef{\ttK}{\widetilde{\tK}} \ndef{\dzK}{{K^\sharp}}
\ndef{\tL}{{\widetilde L}} \ndef{\tcL}{{\widetilde\clL}} \ndef{\ttcL}{\widetilde{\tcL}} \ndef{\sfL}{{\textsf L}} \ndef{\ttL}{\widetilde{\tL}} \ndef{\dzL}{{L^\sharp}}
\ndef{\tM}{{\widetilde M}} \ndef{\tcM}{{\widetilde\clM}} \ndef{\ttcM}{\widetilde{\tcM}} \ndef{\sfM}{{\textsf M}} \ndef{\ttM}{\widetilde{\tM}} \ndef{\dzM}{{M^\sharp}}
\ndef{\tN}{{\widetilde N}} \ndef{\tcN}{{\widetilde\clN}} \ndef{\ttcN}{\widetilde{\tcN}} \ndef{\sfN}{{\textsf N}} \ndef{\ttN}{\widetilde{\tN}} \ndef{\dzN}{{N^\sharp}}
\ndef{\tO}{{\widetilde O}} \ndef{\tcO}{{\widetilde\clO}} \ndef{\ttcO}{\widetilde{\tcO}} \ndef{\sfO}{{\textsf O}} \ndef{\ttO}{\widetilde{\tO}} \ndef{\dzO}{{O^\sharp}}
\ndef{\tP}{{\widetilde P}} \ndef{\tcP}{{\widetilde\clP}} \ndef{\ttcP}{\widetilde{\tcP}} \ndef{\sfP}{{\textsf P}} \ndef{\ttP}{\widetilde{\tP}} \ndef{\dzP}{{P^\sharp}}
\ndef{\tQ}{{\widetilde Q}} \ndef{\tcQ}{{\widetilde\clQ}} \ndef{\ttcQ}{\widetilde{\tcQ}} \ndef{\sfQ}{{\textsf Q}} \ndef{\ttQ}{\widetilde{\tQ}} \ndef{\dzQ}{{Q^\sharp}}
\ndef{\tR}{{\widetilde R}} \ndef{\tcR}{{\widetilde\clR}} \ndef{\ttcR}{\widetilde{\tcR}} \ndef{\sfR}{{\textsf R}} \ndef{\ttR}{\widetilde{\tR}} \ndef{\dzR}{{R^\sharp}}
\ndef{\tS}{{\widetilde S}} \ndef{\tcS}{{\widetilde\clS}} \ndef{\ttcS}{\widetilde{\tcS}} \ndef{\sfS}{{\textsf S}} \ndef{\ttS}{\widetilde{\tS}} \ndef{\dzS}{{S^\sharp}}
\ndef{\tT}{{\widetilde T}} \ndef{\tcT}{{\widetilde\clT}} \ndef{\ttcT}{\widetilde{\tcT}} \ndef{\sfT}{{\textsf T}} \ndef{\ttT}{\widetilde{\tT}} \ndef{\dzT}{{T^\sharp}}
\ndef{\tU}{{\widetilde U}} \ndef{\tcU}{{\widetilde\clU}} \ndef{\ttcU}{\widetilde{\tcU}} \ndef{\sfU}{{\textsf U}} \ndef{\ttU}{\widetilde{\tU}} \ndef{\dzU}{{U^\sharp}}
\ndef{\tV}{{\widetilde V}} \ndef{\tcV}{{\widetilde\clV}} \ndef{\ttcV}{\widetilde{\tcV}} \ndef{\sfV}{{\textsf V}} \ndef{\ttV}{\widetilde{\tV}} \ndef{\dzV}{{V^\sharp}}
\ndef{\tW}{{\widetilde W}} \ndef{\tcW}{{\widetilde\clW}} \ndef{\ttcW}{\widetilde{\tcW}} \ndef{\sfW}{{\textsf W}} \ndef{\ttW}{\widetilde{\tW}} \ndef{\dzW}{{W^\sharp}}
\ndef{\tX}{{\widetilde X}} \ndef{\tcX}{{\widetilde\clX}} \ndef{\ttcX}{\widetilde{\tcX}} \ndef{\sfX}{{\textsf X}} \ndef{\ttX}{\widetilde{\tX}} \ndef{\dzX}{{X^\sharp}}
\ndef{\tY}{{\widetilde Y}} \ndef{\tcY}{{\widetilde\clY}} \ndef{\ttcY}{\widetilde{\tcY}} \ndef{\sfY}{{\textsf Y}} \ndef{\ttY}{\widetilde{\tY}} \ndef{\dzY}{{Y^\sharp}}
\ndef{\tZ}{{\widetilde Z}} \ndef{\tcZ}{{\widetilde\clZ}} \ndef{\ttcZ}{\widetilde{\tcZ}} \ndef{\sfZ}{{\textsf Z}} \ndef{\ttZ}{\widetilde{\tZ}} \ndef{\dzZ}{{Z^\sharp}}

\ndef{\bfa}{{\mathbf a}}
\ndef{\bfb}{{\mathbf b}}
\ndef{\bfc}{{\mathbf c}}
\ndef{\bfd}{{\mathbf d}}

\ndef{\euu}{{\EuScript u}}

  \ndef{\eps}{\varepsilon}

\let\geq\geqslant
\let\leq\leqslant

\ndef{\lims}[1]{\lim\limits_{#1}}
\ndef{\sums}[1]{\sum\limits_{#1}}
\ndef{\ints}[1]{\int_{#1}}
\ndef{\sups}[1]{\sup\limits_{#1}}
\ndef{\liminfty}[1]{\lims{#1\to\infty}}
\ndef{\suminf}[1]{\sums{#1=1}^\infty}

\ndef{\limo}[1]{\omega\mbox{-}\!\!\!\lims{#1\to\infty}}          
\ndef{\limL}[1]{\rmL\mbox{-}\!\!\!\lims{#1\to\infty}}            
\ndef{\limLOne}[1]{\clL_1\mbox{-}\!\lims{#1}}
\ndef{\tildelimo}[1]{\tilde\omega\mbox{-}\!\!\!\lims{#1\to\infty}}
\ndef{\slim}{\mathrm{s}\mbox{-}\!\!\lim}          
\ndef{\wlim}{\mathrm{w}\mbox{-}\!\lim}          

\ndef{\Aut}{\operatorname{Aut}}      
\ndef{\Ch}{\operatorname{ch}}        
\ndef{\End}{\operatorname{End}}      
\ndef{\Hom}{\operatorname{Hom}}      
\rndef{\ker}{\operatorname{ker}}      
\ndef{\coker}{\operatorname{coker}}      
\ndef{\im}{\operatorname{im}}        
\ndef{\Log}{\operatorname{Log}}      
\ndef{\OP}{\operatorname{OP}}        
\ndef{\Op}{\operatorname{Op}}        
\ndef{\Symb}{\operatorname{Symb}}    
\ndef{\Tr}{\operatorname{Tr}}        
\ndef{\Wres}{\operatorname{Wres}}    
\ndef{\cl}{\operatorname{cl}}        
\ndef{\com}{\operatorname{com}}
\ndef{\const}{\operatorname{const}}  
\ndef{\conv}{\operatorname{conv}}    
\rndef{\det}{\operatorname{det}}     

\ndef{\detFK}[1]{\Delta\brs{#1}} 
\ndef{\detFKrel}[2]{\Delta_{#2}\brs{#1}} 

\ndef{\adj}{\operatorname{adj}}    
\ndef{\diag}{\operatorname{diag}}    
\ndef{\dist}{\operatorname{dist}}    
\ndef{\dom}{\operatorname{dom}}      
\ndef{\ec}{\operatorname{ec}}        
\ndef{\id}{1}                        
\ndef{\ind}{\operatorname{ind}}      
\ndef{\mydeg}{\operatorname{deg}}    
\ndef{\op}{\operatorname{op}}
\ndef{\rank}{\operatorname{rank}}
\ndef{\res}{\operatorname{res}}      
\ndef{\rng}{\operatorname{ran}}      
\ndef{\sflow}{\operatorname{sf}}     
\ndef{\isf}{\operatorname{isf}}      
\ndef{\sign}{\operatorname{sign}}    
\ndef{\sgn}{\operatorname{sgn}}      
\ndef{\sing}{\operatorname{sing}}    
\ndef{\supp}{\operatorname{supp}}    
\ndef{\tr}{\operatorname{tr}}        
\ndef{\var}{\operatorname{var}}      
\ndef{\vol}{\operatorname{vol}}      
\ndef{\wn}{\operatorname{wn}}        
\ndef{\wres}{\operatorname{wres}}    
\rndef{\Im}{\operatorname{Im}}       
\rndef{\Re}{\operatorname{Re}}       

\ndef{\prng}[1]{\mathrm R_{#1}} 
\ndef{\pker}[1]{\mathrm N_{#1}} 
\ndef{\rprng}[2]{\mathrm R_{#1}^{#2}}           
\ndef{\rpker}[2]{\mathrm N_{#1}^{#2}}           
\ndef{\rsupp}[1]{\supp_r(#1)}
\ndef{\lsupp}[1]{\supp_l(#1)}
\ndef{\rslv}[1]{R_z(#1)}      
\ndef{\HH}{H}                 
\ndef{\tHH}{\tilde \HH}       
\ndef{\VV}{V}                 
\ndef{\Rz}{R_z}               
\ndef{\tRz}{\tR_z}            
\ndef{\psif}[1]{#1^{[1]}} 
\ndef{\WPlus}[1]{W_{#1}(\mbR)} 

\newcommand{\xia}{\xi^{(a)}}
\newcommand{\xis}{\xi^{(s)}}

\ndef{\bndl}{\xi}                         
\ndef{\bndlA}{\eta}                       
\ndef{\GlueMap}{\varphi}                  
\ndef{\ChartMap}{h}                       
\ndef{\chern}{\ensuremath{\mathrm{ch}}}
\ndef{\hilb}{\clH}                     
\ndef{\hilba}{\clH^{(a)}}                    
\ndef{\hilbs}{\clH^{(s)}}                    
   \ndef{\hilbasargument}{(\hilb)} 
\ndef{\LpH}[1]{\clL_{#1}\hilbasargument}       
\ndef{\saLpH}[1]{\clL_{sa}^{#1}\hilbasargument}       
\ndef{\clBH}{\clB\hilbasargument}              
\ndef{\ubBH}{\clB_1\hilbasargument}            
\ndef{\clCH}{\clC\hilbasargument}              
\ndef{\clKH}{\clK\hilbasargument}              
\ndef{\clFH}{\clF\hilbasargument}              
\ndef{\clUH}{\clU\hilbasargument}              
\ndef{\clCFH}{{\clC\clF}\hilbasargument}       
\ndef{\saBH}{\clB_{sa}\hilbasargument}         
\ndef{\saCH}{\clC_{sa}\hilbasargument}         
\ndef{\saFH}{\clF_{sa}\hilbasargument}         
\ndef{\saKH}{\clK_{sa}\hilbasargument}         
\ndef{\saCFH}{\clC\clF_{sa}\hilbasargument}    
\ndef{\clUFH}{\clU\clF\hilbasargument}         
\ndef{\Uinj}{\clU_{inj}\hilbasargument}        
\ndef{\UFinj}{\clU\clF_{inj}\hilbasargument}   

\ndef{\spproj}[2]{E^{#1}_{#2}}                      
\ndef{\spprojb}[2]{E^{#2}_{#1}}                     

\ndef{\LpN}[1]{\clL^{#1}(\clN,\tau)}     
\ndef{\saLpN}[1]{\clL^{#1}_{sa}(\clN,\tau)} 
\ndef{\rLpN}[1]{L^{#1}(\clN,\tau)}       
\ndef{\clAND}{(\clA,\clN,D)}             
\ndef{\clBA}{{\clB(\clA)}}
\ndef{\saKN}{{\clK_{sa}(\clN,\tau)}}          
\ndef{\clKN}{{\clK(\clN,\tau)}}          
\ndef{\clKtN}{{\clK(\tilde\clN,\tau)}}   
\ndef{\clFN}{{\clF(\clN,\tau)}}          
\ndef{\saFN}{{\clF_{sa}(\clN,\tau)}}     
\ndef{\clPN}{\clP(\clN)}                 
\ndef{\clQN}{\clQ(\clN,\tau)}            
\ndef{\infPN}{{\clP_\tau^\infty(\clN)}}  
\ndef{\clOF}[2]{\clF_{#1\mbox{-}#2}(\clN,\tau)}         
\ndef{\oind}[2]{{\rm \tau\mbox{-}ind}_{#1\mbox{-}#2}}   
\ndef{\tind}{\tau\mbox{-}\ind}                  
\ndef{\DInd}{\ind_{\clD,\tau}}           
\ndef{\BF}{Breuer-Fredholm}              
\ndef{\skewfred}[2]{$(#1\cdot #2)$ $\tau$\tire Fredholm}   
\ndef{\affl}{\eta}                       
\ndef{\vNa}{von Neumann algebra}         
\ndef{\nsf}{faithful normal semifinite } 
\ndef{\taubrs}[1]{\tau\brackets{#1}}     
\ndef{\sqbrs}[1]{[#1]}        
\ndef{\Sqbrs}[1]{\big[#1\big]}        
\ndef{\SqBrs}[1]{\Big[#1\Big]}        

\ndef{\domd}{\bigcap\limits_{n\ge 0} \dom\;\delta^n}         
\ndef{\DiffOP}{{\rm \clD}}
\ndef{\ADA}{\clA \cup [\clD,\clA]}
\ndef{\DixIdeal}[1]{\LpH{#1,\infty}}               
\ndef{\dixideal}{\ell^{1,\infty}}                  
\ndef{\WDixIdeal}{\LpH{1,\mathrm w}}               
\ndef{\DixIdealPos}[1]{\DixIdeal{#1}_+}            
\ndef{\DixIdealN}[1]{\LpN{#1,\infty}}              
\ndef{\DixIdealNPar}[2]{\clL^{#1,\infty}_{#2}(\clN,\tau)}    
\ndef{\DixIdealNPos}[1]{\LpN{#1,\infty}_+}                   
\ndef{\TrD}{\Tr_\omega}                                      
\ndef{\tauD}{{\tau_\omega}}                                  
\ndef{\ILogN}{\frac 1{\log(1+N)}}
\ndef{\DixNorm}[1]{\norm{#1}_{(1,\infty)}}                   
\ndef{\DixInt}[1]{\ints 0^t \mu_s(#1)\,ds}
\ndef{\DixIntL}[1]{\ints 0^{\lambda_{1/t}(#1)}\mu_s(#1)\,ds}
    \ndef{\SmallIdeal}{{\clL^{1, \mathrm w}}}
    \ndef{\SmallIdealMeas}{{\clL^{1, \mathrm w}_m}}
    \ndef{\DixIntII}[1]{\int_0^t \mu_s(#1)\,ds}
    \ndef{\DixIntf}[1]{\Phi_t(#1)}
    \ndef{\DixIntg}[1]{\Psi_t(#1)}

\ndef{\lpi}{\clL^{1,\pi}(\clN,\tau)}

\ndef{\strl}[1]{\stackrel \longrightarrow {#1}}
\ndef{\IIinfty}{$\mathrm{II}_\infty$\ }

\ndef{\fourier}[1]{\clF(#1)}          
\ndef{\HaarMeasBohrs}{\nu}            
\ndef{\BrownsMeas}{\mu}               
\ndef{\BohrCont}[1]{\tilde{#1}}       
\ndef{\APMean}{{M}}                   
\ndef{\CDSS}{{\clA_B}}                
\ndef{\matr}{{\rm Mat}}               
\ndef{\seque}[1]{\ensuremath{\{#1_n\}_{n=1}^\infty}}    
\ndef{\sequen}[2]{\ensuremath{\{#1_#2\}_{#2=1}^\infty}}    
\ndef{\Seque}[1]{\ensuremath{\left(#1_0,#1_1,#1_2,\dots\right)}}    
\ndef{\Cesaro}{H}                           
\ndef{\CesaroRPlus}{M}                      
\ndef{\Dilation}{D}                         
\ndef{\Shift}{T}                            

\ndef{\norm}[1]{\left\Vert#1\right\Vert}    
\ndef{\TrNorm}[1]{\norm{#1}_1}              
\ndef{\HSNorm}[1]{\norm{#1}_2}              
\ndef{\InftyNorm}[1]{\norm{#1}_\infty}      
\ndef{\normQN}[1]{\norm{#1}_{\clQN}}        
\ndef{\clLpnorm}[2]{\norm{#2}_{\clL^{#1}}}    
\ndef{\clLnorm}[1]{\clLpnorm{1}{#1}}    

\ndef{\ccurve}{\gamma}                      

\ndef{\abs}[1]{\left\lvert#1\right\rvert}   
\ndef{\set}[1]{\left\{#1\right\}}           
\ndef{\brackets}[1]{\left(#1\right)}        
\ndef{\brs}[1]{\brackets{#1}}               
\ndef{\Brs}[1]{\big(#1\big)}                
\ndef{\BRS}[1]{\Big(#1\Big)}                
\ndef{\scal}[2]{\left\la #1,#2\right\ra}               
\ndef{\precprec}{\prec\!\!\!\prec}
\ndef{\qeq}{\stackrel?=}
\ndef{\spectrum}[1]{\sigma_{#1}} 
\ndef{\spectruma}[1]{\sigma^{(a)}_{#1}} 
\ndef{\numrange}[1]{\mathrm{W}(#1)}                         
\rndef{\emptyset}{\varnothing}                              
\ndef{\csupp}{c}                           
\ndef{\closure}[1]{\overline{#1}}
\ndef{\linspan}[1]{\mathrm{span}\ {#1}}
\ndef{\bddborel}[1]{B(#1)}                 
\ndef{\charfunc}{\chi}
\ndef{\FrDer}{\euD}                        
\ndef{\LieDer}[1]{\pounds_{#1}\,}          
\ndef{\dds}{\left.\frac d{ds} \right|_{s = 0}}
\ndef{\ortcmp}[1]{#1^{\scriptscriptstyle \perp}}            
\ndef{\Laplace}{\Delta}                    

\ndef{\matrPQ}[3]
{
    \left(
      \begin{array}{cc}
        #1_{11} & #1_{12} \\
        #1_{21} & #1_{22}
      \end{array}
    \right)_{[#2,#3]}
}

\ndef{\margOK}{\marginpar{\bf \small OK}}

\newcounter{margcomcount}
\ndef{\margcom}[1]{\marginpar{\bf \small #1} \addtocounter{margcomcount}{1}
   \index{\indexcom{{\bf COMMENT: #1}}}}

\newcounter{margproof}
\ndef{\margproof}{\marginpar{\bf \small PROOF} \addtocounter{margproof}{1}
  \index{**** \indexcom{{\bf PROOF}}}}

\newcounter{margdetails}
\ndef{\margdetails}{\marginpar{\bf Details} \addtocounter{margdetails}{1}
  \index{**** \indexcom{{\bf DETAILS}}}}

\newcounter{margproofb}
\ndef{\margproofb}[1]{\marginpar{\bf \small Proof(B) #1} \addtocounter{margproofb}{1}
  \index{**** \indexcom{{\bf PROOF(B): #1}}}}

\newcounter{margdetailsb}
\ndef{\margdetailsb}[1]{\marginpar{\bf \small Details(B)} \addtocounter{margdetailsb}{1}
  \index{**** \indexcom{{\bf DETAILS(B): \\ #1}}}}

\newcounter{margdetailsc}
\ndef{\margdetailsc}[1]{\marginpar{\bf \small Details(C)} \addtocounter{margdetailsc}{1}
  \index{**** \indexcom{{\bf DETAILS(C): \\ #1}}}}

\newcounter{margcomcountb}
\ndef{\margcomb}[1]{\marginpar{\bf \small #1} \addtocounter{margcomcountb}{1}
   \index{\indexcom{{\bf COMMENT(B): \\ #1}}}}

\ndef{\mytimes}{\!\times\!}
\ndef{\sss}[1]{\subsubsection{}\label{#1}}

\rndef{\phi}{\varphi}
\ndef{\OpenUnitDisk}{D}
\ndef{\RHS}{RHS}                            
\ndef{\LHS}{LHS} 
\ndef{\ttt}{\Leftrightarrow}
\ndef{\then}{\Rightarrow}
\ndef{\tto}{\longrightarrow}
\ndef{\nno}{\nonumber\\}
\ndef{\newn}[1]{\index{#1} {\bfseries #1}}       
\ndef{\la}{\langle}
\ndef{\ra}{\rangle} \ndef{\dbar}{{\;\bar{\phantom{o}} \!\!\!\! d}}
\ndef{\stl}[1]{\stackrel{\vbox to 0pt{\vss\hbox{$\scriptstyle #1$}}}}
\ndef{\mathcomment}[1]{{\hfill \qquad\qquad\qquad\qquad\qquad\text{by (#1)}}}        
\ndef{\mathcomm}[1]{{\hfill \qquad\qquad\qquad\qquad\qquad\text{#1}}}        
\ndef{\details}[1]{\smallskip\begin{center} {\bf Here:}
#1\end{center}\medskip} \ndef{\indexcom}[1]{ --- #1}
\ndef{\longsim}{\ \sim \ }              
\ndef{\tire}{-}              
\ndef{\intinfinf}{\int_{-\infty}^\infty}
\ndef{\refnsftrace}{\cite[V.\,2.\,1]{TakI}} 
\ndef{\refaffloper}{\cite[IV.\,5, Exercise 3]{TakI}} 
\ndef{\refsemifinvNa}{\cite[V.\,1.\,21]{TakI}} 
\ndef{\reftaumeasurable}{\cite[Definition 1.2]{FK86PJM}} 
\ndef{\reftautraceclassaffl}{\cite[V.2, p.\,320]{TakI}} 
\ndef{\refinvoperideal}{\cite[Appendix A.2]{CP2}} 
\ndef{\reftautracenorm}{\cite[V.2, p.\,320]{TakI}} 
\ndef{\reftaucompact}{\cite{}} 
\ndef{\reftauFredholm}{\cite[Appendix B]{PR94JFA}} 

     \ndef{\npartial}{\slash\!\!\!\partial}
     \ndef{\Heis}{\operatorname{Heis}}
     \ndef{\Solv}{\operatorname{Solv}}
     \ndef{\Spin}{\operatorname{Spin}}
     \ndef{\SO}{\operatorname{SO}}
     \ndef{\Index}{\operatorname{index}}

             \ndef{\p}{\partial}
             \ndef{\dd}{|\clD|}
             \ndef{\n}{\parallel}

     \ndef{\gf}[2]{\genfrac{}{}{0pt}{}{#1}{#2}}
     \ndef{\ta}{\widetilde{\alpha}}
     \ndef{\tb}{\widetilde{\beta}}
     \ndef{\txi}{\widetilde{\xi}}
     \ndef{\tk}{\widetilde{K}}
     \ndef{\CGh}{\widetilde{\CG}}
     \ndef{\boe}{{\bf e}}\ndef{\bt}{{\bf t}}
     \ndef{\vth}{\vartheta}
     \ndef{\db}{\overline{\partial}}
     \ndef{\hV}{\hat{V}}
     \ndef{\cag}{{\clA^\Gamma}}
     \ndef{\sind}{\sigma{\rm -ind}}

\newcounter{slidecount}
\catcode`@=11
\newcommand{\slide}[2]{
  \newpage
  \addtocounter{slidecount}{1}
  \renewcommand{\@oddhead}{{\small Advanced Mathematics 1A -- 2009 \hfil Slide #1-\arabic{slidecount}}}
  \begin{center} \bf #2 \end{center}
}
\catcode`@=12

%% file: MyMakeLit.tex
\let\LatexCite=\cite  

\let\ifnumref\iffalse 

\ndef{\ifuncited}[4]{\expandafter\ifx\csname used#4\endcsname\relax}

\ndef{\ifcited}[4]{\expandafter\ifx\csname used#4\endcsname\relax\else}

  \ndef{\papertitle}[1]{ \emph{#1}, }
  \ndef{\paperauthor}[2]{#2}  
  \ndef{\pbbi}[9]{%
      \ifcited{#1}{#2}{#3}{#5}%
        \ifnumref%
          \bibitem{#5}\paperauthor{#1}{#6},\papertitle{#7}#8.%
        \else%
          \advance #9 by 1%
          \ifnum#9<1%
            \bibitem[#4]{#5}\paperauthor{#1}{#6}, \papertitle{#7}#8.%
          \else%
            \bibitem[#4$_{\the#9}$]{#5}\paperauthor{#1}{#6},\papertitle{#7}#8.%
          \fi%
        \fi%
      \fi%
  }
  \ndef{\mbbi}[8]{%
     \ifcited{#1}{#2}{#3}{#5}%
        \ifnumref%
          \bibitem{#5}\paperauthor{#1}{#6},\papertitle{#7}#8.%
        \else%
          \bibitem[#4]{#5}\paperauthor{#1}{#6},\papertitle{#7}#8.%
        \fi%
     \fi%
  }

\ndef{\AddCite}[1]{%
   \ifuncited{0}{0}{0}{#1}%
     \expandafter\gdef\csname used#1\endcsname {}%
   \fi%
}

\def\ProcessCite#1,{%
     \ifx\relax#1%
         \let\next=\relax%
     \else%
         \AddCite{#1}%
         \let\next=\ProcessCite%
     \fi%
     \next%
}

\ndef{\AddCites}[1]{\ProcessCite#1,\relax,}

\ndef{\CiteWithoutExtension}[1]{%
   \AddCites{#1}%
   \LatexCite{#1}%
}

\def\CiteWithExtension[#1]#2{%
   \AddCites{#2}%
   \LatexCite[#1]{#2}%
}

\ndef{\CleverCite}{%
    \ifx\NChar[ %
       \let\MyCite=\CiteWithExtension %
    \else %
       \let\MyCite=\CiteWithoutExtension %
    \fi %
    \MyCite%
}

\renewcommand{\cite}{\futurelet\NChar\CleverCite}

      \ndef{\volume}[1]{{\bf #1}}
      \ndef{\VolYearPP}[3]{\ifnum#2=0 (to appear)\else\volume{#1} (#2), #3\fi}
      \ndef{\VolNoYearPP}[4]{\ifnum#3=0 (to appear)\else\volume{#1} #2 (#3), #4\fi}
      \ndef{\libcode}[1]{}

\ndef{\jnActaMath}[3]{Acta Math. \VolYearPP{#1}{#2}{#3}}                       
\ndef{\jnAdvMath}[3]{Adv. in~Math. \VolYearPP{#1}{#2}{#3}}                     
\ndef{\jnAlgAnal}[3]{Algebra i~Analiz \VolYearPP{#1}{#2}{#3}}
\ndef{\jnAmerJMath}[3]{Amer. J. Math. \VolYearPP{#1}{#2}{#3}}                  
\ndef{\jnAmerMathMonth}[3]{Amer. Math. Monthly \VolYearPP{#1}{#2}{#3}}         
\ndef{\jnAnnMath}[4]{Ann. of~Math. \VolNoYearPP{#1}{#2}{#3}{#4}}               
\ndef{\jnAnalMath}[3]{J. Anal. Math. \VolYearPP{#1}{#2}{#3}}                   
\ndef{\jnBullLondMathSoc}[3]{Bull. London Math. Soc. \VolYearPP{#1}{#2}{#3}}   
\ndef{\jnBullAMS}[3]{Bull. Amer. Math. Soc. \VolYearPP{#1}{#2}{#3}}   
\ndef{\jnCanMathBull}[3]{Canad. Math. Bull. \VolYearPP{#1}{#2}{#3}}            
\ndef{\jnCanMath}[3]{Canad. J.~Math. \VolYearPP{#1}{#2}{#3}}             
\ndef{\jnCommMathPhys}[3]{Comm. Math. Phys. \VolYearPP{#1}{#2}{#3}}             
\ndef{\jnCommPDE}[3]{Comm. Partial Differential Equations \VolYearPP{#1}{#2}{#3}}             
\ndef{\jnComptRendue}[3]{C.\,R.~Acad. Sci. Paris S\'er. A-B \VolYearPP{#1}{#2}{#3}}      
\ndef{\jnContMath}[3]{Contemporary Math. \VolYearPP{#1}{#2}{#3}}               %
\ndef{\jnDukeMJ}[3]{Duke Math. J. \VolYearPP{#1}{#2}{#3}}
\ndef{\jnDiffGeom}[3]{J.~Diff. Geom. \VolYearPP{#1}{#2}{#3}}                   
\ndef{\jnErgodicTheory}[3]{Ergodic Theory and Dynamical Systems \VolYearPP{#1}{#2}{#3}} 
\ndef{\jnFuncAnal}[3]{J.~Functional Analysis \VolYearPP{#1}{#2}{#3}}           
\ndef{\jnFunkAnalPril}[4]{Funct. Anal. Appl. \VolNoYearPP{#1}{#2}{#3}{#4}}  
\ndef{\jnGAFA}[3]{GAFA \VolYearPP{#1}{#2}{#3}}                                 
\ndef{\jnIHES}[3]{IHES Publ. Math. (Paris) \VolYearPP{#1}{#2}{#3}}             
\ndef{\jnIEOT}[3]{Integral Equations Operator Theory   \VolYearPP{#1}{#2}{#3}} 
\ndef{\jnIsrMath}[3]{Israel J.~Math. \VolYearPP{#1}{#2}{#3}}                   
\ndef{\jnKTheory}[3]{K-Theory \VolYearPP{#1}{#2}{#3}}                          
\ndef{\jnLetMathPhys}[3]{Lett. Math. Phys. \VolYearPP{#1}{#2}{#3}}             
\ndef{\jnMathAnn}[3]{Math. Ann. \VolYearPP{#1}{#2}{#3}}                        
\ndef{\jnMathAnalAppl}[3]{J.~Math. Anal. and Appl. \VolYearPP{#1}{#2}{#3}}     
\ndef{\jnMathNachr}[3]{Math. Nachr. \VolYearPP{#1}{#2}{#3}}
\ndef{\jnMathPhys}[3]{J. Math. Phys. \VolYearPP{#1}{#2}{#3}}
\ndef{\jnMathSocJap}[3]{J. Math. Soc. Japan \VolYearPP{#1}{#2}{#3}}
\ndef{\jnOperTheory}[3]{J.~Operator Theory \VolYearPP{#1}{#2}{#3}}             
\ndef{\jnPacJMath}[3]{Pacific J.~Math. \VolYearPP{#1}{#2}{#3}}                  
\ndef{\jnPositivity}[3]{Positivity \VolYearPP{#1}{#2}{#3}}
\ndef{\jnProcAmerMS}[3]{Proc. Amer. Math. Soc. \VolYearPP{#1}{#2}{#3}}         
\ndef{\jnProcCambPhilSoc}[3]{Math. Proc. Camb. Phil. Soc. \VolYearPP{#1}{#2}{#3}}
\ndef{\jnReineAngew}[3]{J.~Reine Angew. Math. \VolYearPP{#1}{#2}{#3}}          
\ndef{\jnTokyoMath}[3]{Tokyo J.~Math. \VolYearPP{#1}{#2}{#3}}
\ndef{\jnTopology}[3]{Topology \VolYearPP{#1}{#2}{#3}}
\ndef{\jnTransAmerMathSoc}[3]{Trans. Amer. Math. Soc. \VolYearPP{#1}{#2}{#3}}
\ndef{\jnIzvANSSSR}[3]{Izv. Akad. Nauk SSSR, Ser. Mat. \VolYearPP{#1}{#2}{#3}}
\ndef{\jnIzvVyshUchZav}[3]{Izv. Vyssh. Uch. Zav., Mat. \VolYearPP{#1}{#2}{#3} (Russian)}
\ndef{\jnIzdatLenUniv}[2]{Izdat. Leningrad. Univ., Leningrad, (#1), #2 (Russian)}
\ndef{\jnFieldsInsComm}[3]{Fields Inst. Comm. \VolYearPP{#1}{#2}{#3}}
\ndef{\jnDoklANSSSR}[3]{Dokl. Akad. Nauk SSSR \VolYearPP{#1}{#2}{#3}}
\ndef{\jnMatZametki}[3]{Matem. zametki \VolYearPP{#1}{#2}{#3}}
\ndef{\jnRussMathSurvey}[3]{Russian Math. Surveys \VolYearPP{#1}{#2}{#3}}
\ndef{\jnSibMathJ}[3]{Sib. Math.~J. \VolYearPP{#1}{#2}{#3}}
\ndef{\jnSovMath}[3]{J.~Soviet math. \VolYearPP{#1}{#2}{#3}}
\ndef{\jnTransMoscMathSoc}[3]{Trans. Moscow Math. Soc. \VolYearPP{#1}{#2}{#3}}
\ndef{\jnUMN}[3]{Uspekhi Mat. Nauk \VolYearPP{#1}{#2}{#3}}

\ndef{\bkTransMathMon}[2]{Trans. Math. Monographs, AMS, \volume{#1}, #2}

\ndef{\pbBirkhauser}[1]{Birkh\"auser, Boston, #1}
\ndef{\pbFactorial}[1]{Moscow, Factorial, #1}
\ndef{\pbGauthier}[1]{Gauthier-Villars, Paris, #1}
\ndef{\pbNauka}[1]{Moscow, Nauka, #1 (Russian)}
\ndef{\pbNaukaR}[1]{Москва, Наука, #1}
\ndef{\pbPrinceton}[1]{Princeton University Press, Princeton, New Jersey, #1}
\ndef{\pbPublPerish}[1]{Publish or Perish Inc., Berkeley, #1}
\ndef{\pbSpringer}[1]{Springer-Verlag, #1}

\ndef{\myauthor}[1]{\mbox{#1}}

\ndef{\Agmon}{\myauthor{Sh.\,Agmon}}
\ndef{\Ahiezer}{\myauthor{N.\,I.\,Ahiezer}}
\ndef{\Arazy}{\myauthor{J.\,Arazy}}
\ndef{\Aronszajn}{\myauthor{N.\,Aronszajn}}
\ndef{\Astashkin}{\myauthor{S.\,V.\,Astashkin}}
\ndef{\Atiyah}{\myauthor{M.\,Atiyah}}
\ndef{\Avron}{\myauthor{J.\,E.\,Avron}}
\ndef{\Azamov}{\myauthor{N.\,A.\,Azamov}}
\ndef{\Banach}{\myauthor{S.\,Banach}}
\ndef{\Benameur}{\myauthor{M-T.\,Benameur}}
\ndef{\Bennett}{\myauthor{C.\,Bennett}}
\ndef{\Berezin}{\myauthor{F.\,A.\,Berezin}}
\ndef{\Berline}{\myauthor{N.\,Berline}}
\ndef{\Birman}{\myauthor{M.\,Sh.\,Birman}}
\ndef{\Blackadar}{\myauthor{B.\,Blackadar}}
\ndef{\Bogolyubov}{\myauthor{N.\,N.\,Bogolyubov}}
\ndef{\Bonsall}{\myauthor{F.\,F.\,Bonsall}}
\ndef{\Bony}{\myauthor{J.\,F.\,Bony}}
\ndef{\BoosBavnbek}{\myauthor{B.\,Boo$\beta$-Bavnbek}}
\ndef{\Bott}{\myauthor{R.\,Bott}}
\ndef{\Branges}{\myauthor{L.\,de Branges}}
\ndef{\Bratteli}{\myauthor{O.\,Bratteli}}
\ndef{\Bredon}{\myauthor{G.\,E.\,Bredon}}
\ndef{\Breuer}{\myauthor{M.\,Breuer}}
\ndef{\Brown}{\myauthor{L.\,G.\,Brown}}
\ndef{\Bruneau}{\myauthor{V.\,Bruneau}}
\ndef{\Buslaev}{\myauthor{V.\,S.\,Buslaev}}
\ndef{\Carey}{\myauthor{A.\,L.\,Carey}}
\ndef{\CareyRW}{\myauthor{R.\,W.\,Carey}} 
\ndef{\Cartan}{\myauthor{H.\,Cartan}}
\ndef{\Chilin}{\myauthor{V.\,I.\,Chilin}}
\ndef{\Coburn}{\myauthor{L.\,A.\,Coburn}}
\ndef{\Connes}{\myauthor{A.\,Connes}}
\ndef{\Cornfeld}{\myauthor{I.\,P.\,Cornfeld}}
\ndef{\Daletskii}{\myauthor{Yu.\,L.\,Daletski\u\i}}   
\ndef{\Dixmier}{\myauthor{J.\,Dixmier}}
\ndef{\DoddsPG}{\myauthor{P.\,G.\,Dodds}}
\ndef{\DoddsTK}{\myauthor{T.\,K.\,Dodds}}
\ndef{\Douglas}{\myauthor{R.\,G.\,Douglas}}
\ndef{\Dubrovin}{\myauthor{B.\,A.\,Dubrovin}}
\ndef{\Dugundji}{\myauthor{J.\,Dugundji}}
\ndef{\Duncan}{\myauthor{J.\,Duncan}}
\ndef{\Dunford}{\myauthor{N.\,Dunford}}
\ndef{\Dykema}{\myauthor{K.\,J.\,Dykema}}
\ndef{\Edwards}{\myauthor{R.\,E.\,Edwards}}
\ndef{\Eilenberg}{\myauthor{S.\,Eilenberg}}
\ndef{\Entina}{\myauthor{S.\,B.\,\`Entina}}
\ndef{\Fack}{\myauthor{T.\,Fack}} 
\ndef{\Faddeev}{\myauthor{L.\,D.\,Faddeev}}
\ndef{\Farber}{\myauthor{M.\,Farber}}
\ndef{\Farforovskaya}{\myauthor{Yu.\,B.\,Farforovskaya}}
\ndef{\Federer}{\myauthor{H.\,Federer}}
\ndef{\Fedosov}{\myauthor{B.\,V.\,Fedosov}}
\ndef{\Figiel}{\myauthor{T.\,Figiel}} 
\ndef{\Figueroa}{\myauthor{H.\,Figueroa}}
\ndef{\Fillmore}{\myauthor{P.\,A.\,Fillmore}}
\ndef{\Fomenko}{\myauthor{A.\,T.\,Fomenko}} 
\ndef{\Fomin}{\myauthor{S.\,V.\,Fomin}}
\ndef{\Frohlich}{\myauthor{J.\,Fr\"ohlich}}
\ndef{\Fuglede}{\myauthor{B.\,Fuglede}}
\ndef{\Furutani}{\myauthor{K.\,Furutani}}
\ndef{\Gelfand}{\myauthor{I.\,M.\,Gelfand}}
\ndef{\Gesztesy}{\myauthor{F.\,Gesztesy}}     
\ndef{\Getzler}{\myauthor{E.\,Getzler}} 
\ndef{\Gilkey}{\myauthor{P.\,B.\,Gilkey}}
\ndef{\Gitler}{\myauthor{S.\,Gitler}}
\ndef{\Glazman}{\myauthor{I.\,M.\,Glazman}}
\ndef{\Glimm}{\myauthor{J.\,Glimm}}
\ndef{\Gohberg}{\myauthor{I.\,C.\,Gohberg}}
\ndef{\Goldshtein}{\myauthor{Ya.\,Goldshtein}}
\ndef{\Golze}{\myauthor{F.\,Golze}}
\ndef{\GraciaBondia}{\myauthor{J.\,M.\,Gracia-Bond\'{i}a}}
\ndef{\Greenleaf}{\myauthor{F.\,P.\,Greenleaf}}
\ndef{\Gromov}{\myauthor{M.\,Gromov}}
\ndef{\Gunning}{\myauthor{R.\,C.\,Gunning}}
\ndef{\Haagerup}{\myauthor{U.\,Haagerup}}
\ndef{\Haag}{\myauthor{R.\,Haag}}
\ndef{\Halmos}{\myauthor{Halmos}}
\ndef{\Hardy}{\myauthor{G.\,H.\,Hardy}}
\ndef{\Herbst}{\myauthor{I.\,W.\,Herbst}}
\ndef{\Higson}{\myauthor{N.\,Higson}}  
\ndef{\Hoermander}{\myauthor{L.\,Hoermander}} 
\ndef{\Hoffman}{\myauthor{K.\,Hoffman}} 
\ndef{\Ito}{\myauthor{K.\,Ito}}
\ndef{\Jaffe}{\myauthor{A.\,Jaffe}}
\ndef{\James}{\myauthor{I.\,M.\,James}}
\ndef{\Javrjan}{\myauthor{V.\,A.\,Javrjan}}
\ndef{\Jitomirskaya}{\myauthor{S.\,Jitomirskaya}}
\ndef{\Kadison}{\myauthor{R.\,V.\,Kadison}}
\ndef{\Kalton}{\myauthor{N.\,J.\,Kalton}} 
\ndef{\Kato}{\myauthor{T.\,Kato}} 
\ndef{\Kobayashi}{\myauthor{S.\,Kobayashi}}
\ndef{\Koplienko}{\myauthor{L.\,S.\,Koplienko}}
\ndef{\Korotyaev}{\myauthor{E.\,Korotyaev}}
\ndef{\Kosaki}{\myauthor{H.\,Kosaki}}
\ndef{\Kostrykin}{\myauthor{V.\,Kostrykin}}
\ndef{\Kotani}{\myauthor{S.\,Kotani}}
\ndef{\Krein}{\myauthor{Kre\u\i n}}
\ndef{\KreinMG}{\myauthor{M.\,G.\,Kre\u\i n}}
\ndef{\KreinSG}{\myauthor{S.\,G.\,Kre\u\i n}}
\ndef{\Kuroda}{\myauthor{S.\,T.\,Kuroda}}
\ndef{\Leichtnam}{\myauthor{E.\,Leichtnam}}
\ndef{\Lesch}{\myauthor{M.\,Lesch}}
\ndef{\Lesniewski}{\myauthor{A.\,Lesniewski}}
\ndef{\Levitan}{\myauthor{B.\,M.\,Levitan}}
\ndef{\Lidskii}{\myauthor{V.\,B.\,Lidskii}}
\ndef{\Lifshits}{\myauthor{I.\,M.\,Lifshits}}
\ndef{\Lindenstrauss}{\myauthor{J.\,Lindenstrauss}}
\ndef{\Loday}{\myauthor{J.-L.\,Loday}}
\ndef{\Lord}{\myauthor{S.\,Lord}}      
\ndef{\Lorentz}{\myauthor{G.\,Lorentz}}
\ndef{\Magnus}{\myauthor{W.\,Magnus}}
\ndef{\Makarov}{\myauthor{K.\,A.\,Makarov}}
\ndef{\MakarovN}{\myauthor{N.\,Makarov}}
\ndef{\Mathai}{\myauthor{V.\,Mathai}}         
\ndef{\McKean}{\myauthor{H.\,P.\,McKean}}
\ndef{\Mishchenko}{\myauthor{A.\,S.\,Mishchenko}}
\ndef{\Molchanov}{\myauthor{S.\,A.\,Molchanov}}
\ndef{\Moore}{\myauthor{C.\,C.\,Moore}}
\ndef{\Moscovici}{\myauthor{H.\,Moscovici}}  
\ndef{\Motovilov}{\myauthor{A.\,K.\,Motovilov}}
\ndef{\Moyer}{\myauthor{R.\,D.\,Moyer}}
\ndef{\Naboko}{\myauthor{S.\,N.\,Naboko}}
\ndef{\Narasimhan}{\myauthor{R.\,Narasimhan}}
\ndef{\Nomizu}{\myauthor{K.\,Nomizu}}
\ndef{\Novikov}{\myauthor{S.\,P.\,Novikov}}
\ndef{\Osterwalder}{\myauthor{K.\,Osterwalder}}
\ndef{\Patodi}{\myauthor{V.\,Patodi}}
\ndef{\Pagter}{\myauthor{B.\,de~Pagter}}  
\ndef{\Pastur}{\myauthor{L.\,A.\,Pastur}}  
\ndef{\Pavlov}{\myauthor{B.\,S.\,Pavlov}}
\ndef{\Pedersen}{\myauthor{G.\,K.\,Pedersen}}
\ndef{\Peller}{\myauthor{V.\,V.\,Peller}}
\ndef{\Perera}{\myauthor{V.\,S.\,Perera}}
\ndef{\Petunin}{\myauthor{Ju.\,I.\,Petunin}}
\ndef{\Phillips}{\myauthor{J.\,Phillips}}  
\ndef{\Piazza}{\myauthor{P.\,Piazza}}   
\ndef{\Pincus}{\myauthor{J.\,D.\,Pincus}}   
\ndef{\Poincare}{Poincar\'e}
\ndef{\Postnikov}{\myauthor{M.\,M.\,Postnikov}} 
\ndef{\Prinzis}{\myauthor{R.\,Prinzis}}
\ndef{\Privalov}{\myauthor{I.\,I.\,Privalov}}
\ndef{\Pushnitski}{\myauthor{A.\,B.\,Pushnitski}} 
\ndef{\Raeburn}{\myauthor{I.\,Raeburn}}
\ndef{\Raikov}{\myauthor{G.\,Raikov}}
\ndef{\Reed}{\myauthor{M.\,Reed}}
\ndef{\Rennie}{\myauthor{A.\,Rennie}}
\ndef{\Rickart}{\myauthor{C.\,E.\,Rickart}}
\ndef{\Riesz}{\myauthor{F.\,Riesz}}
\ndef{\Ringrose}{\myauthor{J.\,Ringrose}}
\ndef{\Rio}{\myauthor{R.\,del Rio}}
\ndef{\Robinson}{\myauthor{D.\,Robinson}}
\ndef{\Rossi}{\myauthor{H.\,Rossi}}
\ndef{\Rudin}{\myauthor{W.\,Rudin}}
\ndef{\Ruelle}{\myauthor{D.\,Ruelle}}
\ndef{\Ruzhansky}{\myauthor{M.\,Ruzhansky}}
\ndef{\Sakai}{\myauthor{Sh.\,Sakai}}
\ndef{\Sargsjan}{\myauthor{I.\,S.\,Sargsjan}}
\ndef{\Sato}{\myauthor{H.\,Sato}}
\ndef{\Schaeffer}{\myauthor{D.\,G.\,Schaeffer}}
\ndef{\Schluchtermann}{\myauthor{G.\,Schluchtermann}}
\ndef{\Schochet}{\myauthor{C.\,Schochet}}
\ndef{\SchroedingerE}{\myauthor{E.\,Schr\"odinger}}
\ndef{\Schroedinger}{\myauthor{Schr\"odinger}}
\ndef{\Schrohe}{\myauthor{E.\,Schrohe}}
\ndef{\Schwartz}{\myauthor{J.\,T.\,Schwartz}}
\ndef{\Sedaev}{\myauthor{A.\,A.\,Sedaev}}
\ndef{\Seiler}{\myauthor{R.\,Seiler}}
\ndef{\Semenov}{\myauthor{E.\,M.\,Semenov}}
\ndef{\Shabat}{\myauthor{B.\,V.\,Shabat}}
\ndef{\Shafarevich}{\myauthor{I.\,R.\,Shafarevich}}
\ndef{\Sharpley}{\myauthor{R.\,Sharpley}}
\ndef{\Shilov}{\myauthor{G.\,E.\,Shilov}}
\ndef{\Shirkov}{\myauthor{D.\,V.\,Shirkov}}
\ndef{\Shubin}{\myauthor{M.\,A.\,Shubin}}
\ndef{\Silverman}{\myauthor{H.\,Silverman}}
\ndef{\Simon}{\myauthor{B.\,Simon}}
\ndef{\Sinai}{\myauthor{Ya.\,G.\,Sinai}}
\ndef{\Singer}{\myauthor{I.\,M.\,Singer}}
\ndef{\Solomyak}{\myauthor{M.\,Z.\,Solomyak}}
\ndef{\Soloviev}{\myauthor{Yu.\,P.\,Soloviev}}
\ndef{\Spivak}{\myauthor{M.\,Spivak}}
\ndef{\Stein}{\myauthor{E.\,M.\,Stein}}
\ndef{\Stenkin}{\myauthor{V.\,V.\,Sten'kin}}
\ndef{\Stratila}{\myauthor{S.\,Stratila}}
\ndef{\Sucheston}{\myauthor{L.\,Sucheston}}
\ndef{\Sukochev}{\myauthor{F.\,A.\,Sukochev}}
\ndef{\Switzer}{\myauthor{R.\,M.\,Switzer}}
\ndef{\SzNagy}{\myauthor{B.\,Sz.-Nagy}}
\ndef{\Takesaki}{\myauthor{M.\,Takesaki}}
\ndef{\Taylor}{\myauthor{M.\,E.\,Taylor}}
\ndef{\Treves}{\myauthor{F.\,Treves}}
\ndef{\Troitsky}{\myauthor{E.\,V.\,Troitsky}}
\ndef{\Tzafriri}{\myauthor{L.\,Tzafriri}}
\ndef{\Varilly}{\myauthor{J.\,C.\,V\'{a}rilly}}
\ndef{\Vergne}{\myauthor{M.\,Vergne}}
\ndef{\Vladimirov}{\myauthor{V.\,S.\,Vladimirov}}
\ndef{\Voiculescu}{\myauthor{D.\,Voiculescu}}
\ndef{\Weiss}{\myauthor{G.\,Weiss}}
\ndef{\Wells}{\myauthor{R.\,O.\,Wells}}
\ndef{\Williams}{\myauthor{J.\,P.\,Williams}}
\ndef{\Winkler}{\myauthor{S.\,Winkler}}
\ndef{\Witten}{\myauthor{E.\,Witten}}
\ndef{\Wodzicki}{\myauthor{M.\,Wodzicki}}
\ndef{\Wojciechowski}{\myauthor{K.\,P.\,Wojciechowski}}
\ndef{\Yafaev}{\myauthor{D.\,R.\,Yafaev}}
\ndef{\Yosida}{\myauthor{K.\,Yosida}}
\ndef{\Zsido}{\myauthor{L.\,Zsido}}


%% file: MyListOfRef.tex

\rndef{\emph}[1]{{\it #1}}

\mathsurround 0pt
\ndef{\AndSoOn}{$\dots$}
